\documentclass[12pt]{article}
\usepackage{amsmath}
\usepackage{amssymb}
\usepackage{enumerate}
\usepackage{amsthm}
\usepackage{mathrsfs}
\usepackage{footnote}
\usepackage{graphicx}
\usepackage{setspace}
\usepackage{caption}
\usepackage[margin=1in]{geometry}

\numberwithin{equation}{section}

\newcommand{\R}{\mathbb{R}}

\newcommand{\N}{\mathbb{N}}
\newcommand{\C}{\mathcal{O}(1)}
\newcommand{\T}{\mathbb{T}}

\newtheorem{lemma}{Lemma}

\newtheorem{thm}{Theorem}[section]
\theoremstyle{definition}
\newtheorem{defn}[thm]{Definition}
\newtheorem{remark}{Remark}[thm]

\begin{document}

\date{\today}
\title{Borel Cantelli Lemmas and Extreme Value Theory for Geometric Lorenz Models}

\author{Licheng Zhang\thanks{The research is part of my thesis and it was partially supported by the NSF grant DMS 1101315. I would like to thank Matthew Nicol for discussion, comments and advice. I also wish to thank S. Vaienti and M. Holland.}}

\maketitle

\begin{abstract}

We establish dynamical  Borel-Cantelli lemmas  for nested balls and rectangles centered at generic points in the setting of geometric Lorenz maps. We also 
establish extreme value statistics for observations maximized at generic points for geometric Lorenz maps and  the associated flow.
\end{abstract}

\tableofcontents{}

\section{Introduction}
In a chaotic system, the future behavior of the system is very sensitive to the initial conditions and so a statistical
description of the system's behavior is often the most appropriate.  
We may investigate whether  suitable versions of classical limit  theorems from probability theory such as law of large numbers, central limit theorem, Borel-Cantelli lemma, extreme value theory and so on,
hold and use this knowledge to make predictions about the system's behavior. In this paper, we study a particular system, which is Lorenz system, and establish Strong Borel Cantelli lemma and Extreme Value Laws for it. 


The equations defining the Lorenz system were first published in the \textit{Journal of Atmospheric Sciences}(\cite{Lorenz}) as a parametrized polynomial system of differential equations:
\begin{align*}
  \dot{x} &= \sigma (y-x)        \\
  \dot{y} &= x(\rho - z) - y     \\
  \dot{z} &= xy -\beta z         \\
\end{align*}
where $\sigma = 10$, $\rho = 28$, $\beta = 8/3$. The system was proposed  as a simplified model for thermal fluid convection, 
motivated by a desire to understand   weather systems. 
What is interesting is that the equations are deterministic but they produce chaotic behavior, with trajectories spiraling around two attractors seemingly randomly. 
\begin{figure}[htbp]
 \centering
 \includegraphics[width= 0.65\textwidth]{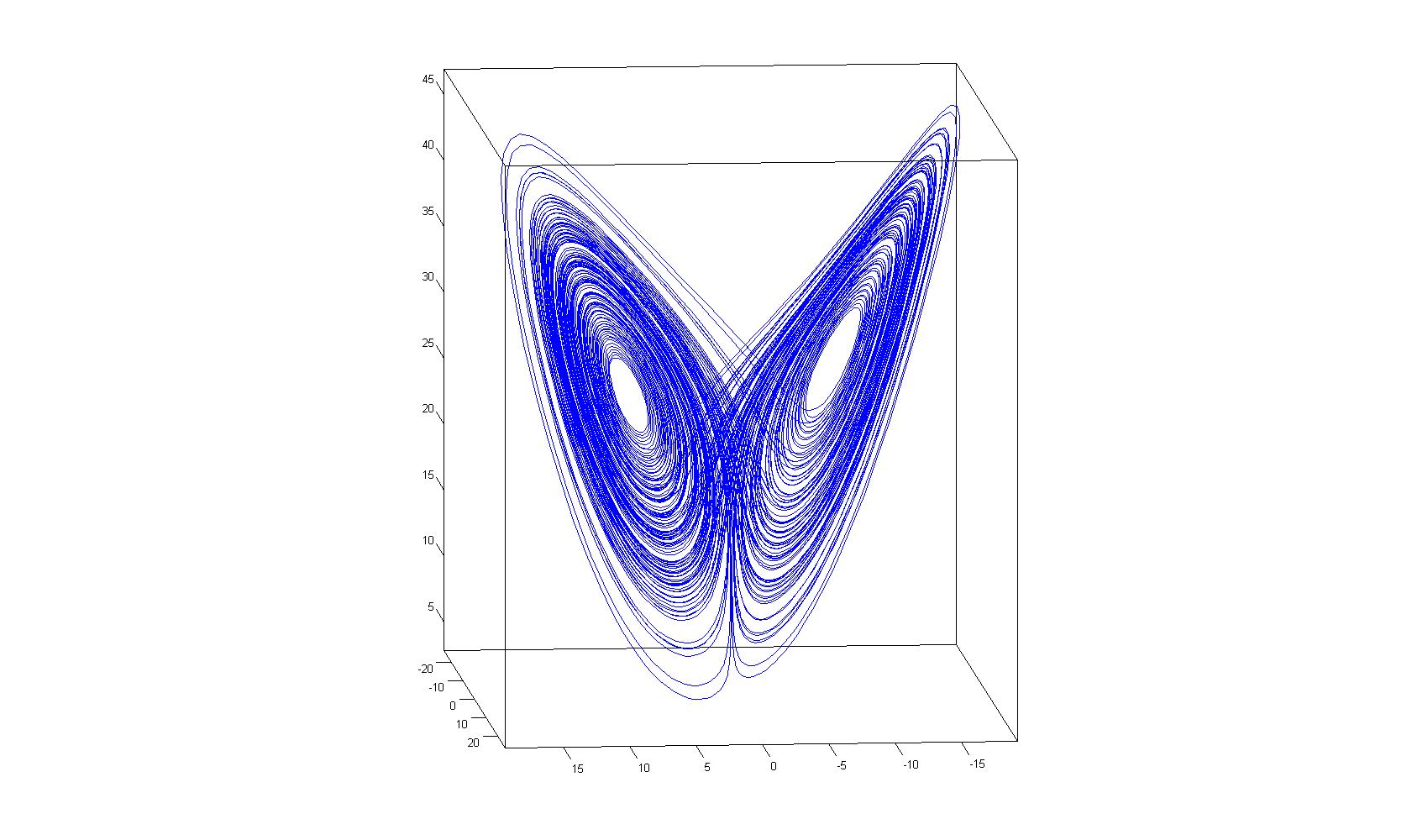} 
\centering
\caption{Lorenz attractor}
\end{figure}
In order to achieve insights  on this system, 
a very successful approach was taken by Afraimovich, Bykov and Shil'nikov\cite{ABS}, and Guckenheimer, Williams\cite{GW}, independently: they constructed the so-called \textsl{Geometric Lorenz models}.  
These models are flows in three-dimension which have properties very similar to the Lorenz systems and are easier to study. One can rigorously prove the existence of an attractor that contains an equilibrium point of the flow, together with regular solutions. The original proof of the existence of a chaotic attractor was made by Warwick Tucker in the year 2000, with the help of computer (see \cite{Tucker1, Tucker2}).

Here we give a brief version of construction of the Geometric Lorenz model, and more detailed version can be found in \cite[section 2.1]{GP}. Consider a linear system in $[-1, 1]^3$:
\[
 (\dot{x}, \dot{y}, \dot{z}) = (\lambda_1 x, \lambda_2 y, \lambda_3 z)
\]
with $\lambda_1$, $\lambda_2$, $\lambda_3$ satisfying
\[
 0<\frac{\lambda_1}{2} \le - \lambda_3 < \lambda_1 < -\lambda_2
\]
For any initial point $(a,b,c) \in \R^3$ near the equilibrium $(0,0,0)$, the trajectories are given by 
\[
 \tilde{L}_t (a,b,c) = (a e^{\lambda_1 t}, b e^{\lambda_2 t},c e^{\lambda_3 t})
\]
where $\tilde{L}_t$ denotes the linear flow.  

Consider $\Omega = \{ (x,y,1) : |x| \le \frac{1}{2},~ |y| \le \frac{1}{2} \} = \Omega^- \cup \Omega^o \cup \Omega^+$, where 
\[
 \Omega^- = \{ (x, y, 1) \in \Omega: x<0 \}
\]
\[
 \Omega^+ = \{ (x, y, 1) \in \Omega: x>0 \}
\]
\[
 \Omega^o = \{ (x, y, 1) \in \Omega: x=0 \}
\]
$\Omega$ is a transverse section to the linear flow $\tilde{L}_t$, and since $\lambda_3 < 0$ , every trajectory, that would cross $\Omega$, will cross in the direction of the negative $z$ axis. Let $\Omega^* = \Omega^- \cup \Omega^+$ and let $\tilde{\Omega} = \{ (x,y,z): |x| = 1\} = \tilde{\Omega}^- \cup \tilde{\Omega}^+$ with $\tilde{\Omega}^{\pm} = \{(x,y,z): x =\pm 1\}$. For each $(a,b,1) \in \Omega^*$, the time $t$ such that $\tilde{L}_t(a,b,1) \in \tilde{\Omega}$ is given by 
\[
|a e^{\lambda_1 t}| = 1 ~\Longrightarrow~ t(a) = -\frac{1}{\lambda_1}\log |a|
\]
the time only depends on the first component of the point in $\Omega^*$ and $t(a) \to \infty$ as $a \to 0$. Thus, we can express the point in $\tilde{\Omega}$ mapped from point $(a,b,1) \in \Omega^*$ explicitly:
\[
\tilde{L}_{t(a)}(a,b,1) = (sgn(a), b e^{\lambda_2 t(a)}, e^{\lambda_3 t(a)})= (sgn(a), b |a|^{-\frac{\lambda_2}{\lambda_1}}, |a|^{-\frac{\lambda_3}{\lambda_1}})
\]
where $sgn(a) = a/|a|$ for $a \neq 0$. In this way, we just defined a map $L: \Omega^* \to \tilde{\Omega}^\pm $ by 
\[
L(x,y,1) = (sgn(x), y|x|^\beta, |x|^\alpha)
\]
where $\beta = -\frac{\lambda_2}{\lambda_1}$, $\alpha = -\frac{\lambda_3}{\lambda_1}$ satisfying 
$\frac{1}{2} < \alpha <1<\beta$, since $0<\frac{\lambda_1}{2} \le - \lambda_3 < \lambda_1 < -\lambda_2$.

Then we should let the sets $L(\Omega^*)$ return to the cross section $\Omega$ through a flow defined by a suitable composition of a rotation $R_\pm$, an expansion $E_{\pm \theta}$ and a translation $\T_\pm$. More precisely, for $(x,y,z) \in \tilde{\Omega}^\pm$,
\[
R_\pm (x,y,z)=\begin{pmatrix}
0 & 0 & \pm 1 \\ 0 & 1 & 0 \\ \pm 1 & 0 & 0
\end{pmatrix}
\]
and for $(x,y,z) \in \Omega$,
\[
E_{\pm \theta}(x,y,z) = \begin{pmatrix}
\theta & 0 & 0 \\
0 & 1 & 0 \\
0 & 0 & 1
\end{pmatrix} .
\]
$\theta$ and $\T_\pm$ shall be chosen to satisfy certain conditions.

So the Poincar\'e first return map, i.e. our Lorenz map, $F: \Omega^* \to \Omega$, is defined as 
\[
F(x,y) = 
\begin{cases}
\T_+ \circ E_{+\theta} \circ R_+ \circ L(x,y,1)& \mbox{for } x> 0\\
\T_- \circ E_{-\theta} \circ R_- \circ L(x,y,1)& \mbox{for } x< 0\\
\end{cases}
\]
Combining  the effect of the rotation with expansion and translation, $F$ must have the form: 
\[
F(x,y) = (T(x), G(x,y))
\] 
where $T:I\backslash \{0\} \to I$ and $G: (I\backslash\{0\}) \times I \to I$, where $I=[-\frac{1}{2}, \frac{1}{2}]$. 
Here $T$ is given by 
\[
T(x) = \begin{cases}
f_1(x^{\alpha}), & x<0 \\
f_0(x^{\alpha}), & x>0
\end{cases}
\]
with $f_i(x)= (-1)^i \theta \cdot x +b_i, i\in \{ 0,1\}$ such that $\theta \cdot (\frac{1}{2})^\alpha <1$ and $\theta \cdot \alpha \cdot 2^{1-\alpha} >1$. $T$ is the quotient map of $F$, 
usually referred to as the Lorenz like map, see Figure \ref{lorenzlike}. 
\begin{figure}[htbp]
\centering
\includegraphics[width=0.4\textwidth]{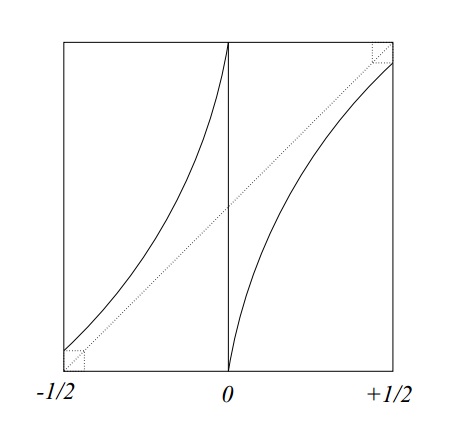}
\caption{Lorenz like map $T$}\label{lorenzlike}
\end{figure}
It has the following properties:
\begin{enumerate}
\item $T$ is discontinuous at $x = 0$, and the lateral limits $T(0^\pm)$ do exist, $T(0^\pm) = \mp \frac{1}{2}$.
\item $T$ is $C^2$ on $I\backslash \{0\}$ and $T'(x) > 1$ for all $x \in I\backslash \{0\}$.
\item $\lim\limits_{x \to 0} T'(x) = +\infty$
\end{enumerate}

And G is given by 
\[
G(x,y) = \begin{cases}
g_1(x^\alpha, y \cdot x^\beta), & x<0 \\
g_0(x^\alpha, y \cdot x^\beta), & x>0
\end{cases}
\]
where $g_1 \vert I^- \times I \to I$ and $g_0 \vert I^+ \times I \to I$ are suitable affine maps. Here $I^- = (-1/2,0)$, $I^+=(1/2,0)$.

Much recent work has focused on the ergodic and statistical properties  of Lorenz like maps
including rates of mixing, extreme value theory and  return time statistics.  S. Galatolo and M.J. Pacifico \cite{GP} proved that the Poincar\'e map, i.e. our Lorenz map $F$, 
associated to a Lorenz like flow has exponential decay of correlations with respect to Lipschitz observables and the hitting time statistics satisfies a logarithm law.

\subsection{Local dimension}\label{localdim}

 Let $(M,d)$ be a metric space and assume that $\mu$ is a Borel probability measure on $M$. Given $x \in M$, let $B_r(x) = \{ y \in M: d(x,y) \le r \}$ be the ball centered at  $x$ 
with radius $r$. The {\sl local dimension} of $\mu$ at $x \in M$ is defined by 
\[
 d_{\mu} (x)= \lim_{r \to 0} \frac{\log \mu(B_r(x))}{\log r}
\]
if this limit exists. In this case $\mu(B_r(x)) \sim r^{d_\mu (x)}$.

A result of Afraimovich and Pesin~\cite[Theorem 9]{AP} ensures that for the  Lorenz system, the local dimension exists and is constant for $\mu$ a.e. point.  

\subsection{Young Tower Structure and Local Product Structure}\label{localstr}

A recent paper of Araujo, Melbourne and Varandas~\cite{AMV} used a Young Tower construction to establish that a broad range of geometric Lorenz flows are rapidly mixing. Along the way they
showed that $\mu$ a.e. $p \in M$ has a local product structure (this follows from their Proposition 2.4). More precisely, $\mu$ a.e. $p$ has the property that there exists a $r(p)>0$ such that for all $r<r(p)$
if $A(r)$ is a square of  sidelength $2r$, or a ball of radius $r$, centered at $p$, then $\mu$ a.e. $q\in A(r)$ has a local unstable manifold and a local stable manifold  which fully crosses $A(r)$. Moreover
if $q_1$, $q_2$ are in $A(r)$ then there is a unique point $z=W^{u}_{loc} (q_1) \cap W^{s}_{loc} (q_2)\in A(r)$.

\subsection{Short returns for one-dimensional  Lorenz-like maps}

Gupta, Holland and Nicol~\cite{GHN} established extreme value statistics for Lorenz-like maps. The proofs used a crucial estimate on the measure of points with 
short returns. In particular, on page 21 they showed that for $\mu$ a.e. $p\in I$, and for all sufficiently small $r<r(p)$, if $B_r(p)$ is  a ball of radius $r$ based at $p$, then there are constants $C>0$, $0<\alpha<1$ 
such that for all $1\le j \le (\log r)^5$, $\mu (B_r \cap T^{-j} B_r) \le \mu (B_r) e^{-(\log r)^{\alpha}}$.

\subsection{Borel Cantelli Lemma}
The classical Borel-Cantelli lemmas are as follows: suppose $(\Omega, \mathcal{B}, \mu)$ is a probability space. Let $\textbf{1}_{A}$ be the characteristic function of $A$, given $A$ is a measurable set of $\Omega$. Then
\begin{enumerate}
 \item if $(A_n)_{n=0}^{\infty}$ is a sequence of measurable sets in $\Omega$ and $\sum_{n=0}^\infty \mu(A_n) <\infty$, then $\mu(x \in A_n~ i.o.)=0$
 \item if $(A_n)_{n=0}^{\infty}$ is a sequence of independent sets in $\Omega$ and  $\sum_{n=0}^\infty \mu(A_n) =\infty$, then for $\mu$ a.e. $x\in \Omega$
 \[
 \frac{S_n(x)}{E_n} \to 1
 \]
 where $S_n(x) = \sum_{j=0}^{n-1} \textbf{1}_{A_j}(x)$ and $E_n = \sum_{j=0}^{n-1} \mu(A_j)$.
\end{enumerate}
In the dynamical systems setting, $T:\Omega \to \Omega$ is usually considered to be a measure-preserving transformation of 
the probability space $(\Omega, \mathcal{B}, \mu)$. Suppose that $(A_n)_{n=0}^{\infty}$ is a sequence of sets in $\mathcal{B}$ 
such that $\sum_{n=0}^\infty \mu(A_n) =\infty$.  Let $E_n = \sum_{j=0}^{n-1} \mu(A_j)$ and 
$S_n(x) = \sum_{j=0}^{n-1} \textbf{1}_{A_j} \circ T^j (x) $. Then we call the sequence $(A_n)$: 
\begin{enumerate}
\item {\bf a Borel Cantelli sequence(BC)} if $\mu(x: T^n x \in A_n~ i.o.)=1$, i.e. $S_n(x)$ is unbounded.
\item {\bf a Strong Borel Cantelli sequence(SBC)} if  $\lim\limits_{n\to \infty}\frac{S_n(x)}{E_n} = 1, ~a.s.$.
\end{enumerate}

\begin{remark}
 If the sequence $(A_n)_{n=0}^{\infty}$ are nested balls of radius $r_n$ about a point $p$ in the dynamical system $(T, \Omega, \mu)$, then the question of whether $T^i x \in A_i$ infinitely often for $\mu$ a.e. $x$ is called the shrinking target problem. For the rest of the paper, we establish strong Borel Cantelli Lemmas and Extreme Value Laws for the Lorenz maps $F$ regards to shrinking target property. 
\end{remark}


\subsection{Extreme Value Laws(EVL)}\label{evl}


We consider a dynamical system $(\Omega, \mathcal{B}, \mu, F)$ where $F$ preserves  an invariant measure $\mu$. Consider the time series $X_0, X_1, X_2 \cdots$ arising 
from this system  by evaluating a given random variable(r.v.) $\varphi: \Omega \rightarrow 
\mathbb{R} \cup 
\{\pm \infty \}$ along the orbits of the system, that is to say, we define
\begin{equation}
\label{proc}
 X_n = \varphi \circ F^n
\end{equation}
for each $n \in \mathbb{N}$. Apparently, $X_0, X_1, \cdots$ defined in this way is not an
independent sequence, but $F$-invariance of $\mu$ guarantees that the stochastic process is stationary.

Here we suppose that $\varphi$ has one global maximum at $\zeta\in \Omega$ ($\varphi(\zeta)=+\infty$ is allowed).  And let $u_F:=\varphi(\zeta)$. By assuming that $\varphi$ and $\mu$ are sufficiently regular,  the event 
\begin{equation*}
U(u):=\{x\in \Omega:\; \varphi(x)>u\}=\{X_0>u\}
\end{equation*} 
corresponds to a topological ball centered at $\zeta$ for $u$ sufficiently close to $u_F$. Furthermore, $\mu(U(u))$ varies continuously as a function of $u$ on a neighbourhood of $u_F$.

\begin{defn}(Logarithmic singularity)
 Consider a function $\varphi: \Omega \to \R$ and a point $x_0 \in \Omega$. Let $d$ be a distance function on $\Omega$. We say that $\varphi$ has a logarithmic singularity at the point $x_0$ if $x_0$ has a neighborhood where $\varphi(x) = - C \log d(x,x_0) + g(x)$ with $C > 0$ , where $g$ is bounded and has a finite limit as $x \to x_0$. 
\end{defn}

We are always interested in studying the extremal behavior of the stochastic process $X_0, X_1,\ldots$, and it is associated with the occurrence of exceedances of high levels $u$. 
 When $u$ is close to $u_F$, the occurrence of the event $\{X_j>u\}$ means that  the occurrence of an exceedance at time $j\in\mathbb{N}_0$. This is equivalent to saying that $F^j(x)\in U(u)$, i.e. the orbit of the point $x$ hits the ball $U(u)$ at time $j$.

In order to consider the extremal behavior of the system for which we define a new sequence of random variables $M_1, M_2, \cdots$ given by
\[
M_n = \max\{X_0, \cdots,X_{n-1} \}
\]

\begin{defn}(Extreme Value Laws)
We say that we have an $EVL$ for $M_n$ if there is a non-degenerate distribution function(d.f.) $G: \R \to [0,1]$ with $G(0)=0$ and, for every $\upsilon >0$,
there exists a sequence of levels $u_n = u_n(\upsilon)$, $n=1,2,\cdots$, such that
\begin{equation}
\lim_{n\to\infty}n\mu(X_0>u_n)=\upsilon \label{rate}
\end{equation}
and for which the following holds:
\[
\mu(M_n \le u_n) \to \bar{G}(\upsilon) \text{ as $n \to \infty$}
\]
where $\bar{G} = 1- G$. 
\end{defn}

The motivation for using such normalising sequences \eqref{rate} comes from the case when $X_0, X_1,\ldots$ are independent and identically distributed (i.i.d.). 
In this i.i.d.\ setting, it is clear that $P(M_n\leq u)= (Z(u))^n$, where $Z$ is the distribution function\ of $X_0$, i.e. $Z(x):=P(X_0\leq x)$. Therefore, condition \eqref{rate} indicates that
\[
P(M_n\leq u_n)= (1-P(X_0>u_n))^n\sim\left(1-\frac\upsilon n\right)^n\to e^{-\upsilon},
\]
as $n\to\infty$. This implies that, approximately, the waiting times between exceedances of $u_n$ is exponentially distributed. Moreover, the reciprocal is also true.
Note that in this case $G(\upsilon)= 1 - e^{-\upsilon}$ is the standard exponential d.f..

\vspace{.5cm}

\begin{remark}
 We will give result on Lorenz flows $f_t$, in which case we consider continuous time stochastic process $\{X_t\}$ and define the process of successive maxima $M_T:= \sup_{0\le t \le T} \{X_t\}$.
\end{remark}

\begin{remark}
For independent and identically distributed(i.i.d.) processes, there are only three possible types of non-degenerate extremal distributions (subject to linear scaling):
\begin{itemize}
\item Type I    $$G(x) = e^{-e^{-x}},  -\infty < x < \infty$$
\item Type II   $$G(x) = \begin{cases}
0   & \mbox{if } x<0; \\
e^{-x^{-\alpha}} & \mbox{for some $\alpha$ if } x>0 .
\end{cases}
$$
\item Type III
\[
G(x) = \begin{cases}
e^{-(-x)^\alpha} & \mbox{for some $\alpha >0$ if } x<0;\\
1 & \mbox{if } x>0.
\end{cases}
\]
\end{itemize}
For dependent stationary processes $\{X_n\}$, $EVL$ will satisfy under $D_3(u_n)$ and $D'(u_n)$, which will be introduced in the next subsection. 
\end{remark}

For $\mu$ a.e. $x_0$, if we consider
\begin{equation}
 \label{observable}
\varphi(x) = - \log d(x, x_0)
\end{equation}
where $d(\cdot, \cdot)$  is the local metric on $\Omega$ and $x_0$ is a fixed point.  And define $U_n = \{X_0 > u_n \}$, where $u_n = u_n(v)$ such that $\mu(U_n) = e^{-v}/n$.  Here $u_n$ is an increasing sequence going to $\varphi(x_0)$(which is $+\infty$) and assume $U_n$ corresponds to a topological ball centered at $x_0$ with radius $e^{-u_n}$. Then the corresponding processes $\{ X_n\}$ will satisfy Type I extremal distribution.  

To be consistent, if not specified,  we will use the defintion and assuption above for section on EVL.
\begin{remark}
We define a function that we refer to as \textsl{first hitting time function} to a set $A\in \mathcal{B}$, and denote by $r_A:\Omega \to \N 
\cup \{\infty\}$ where
\[
r_A(x) = \min\{j \in \N \cup \{\infty\}: F^j(x) \in A\}
\]
The restriction of $r_A$ to $A$ is called the \textsl{first return time function} to $A$, denoted by $R(A)$, as the minimum of the return time function to $A$, $i.e.$
\[
R(A) = \min_{x\in A}r_A(x)
\]
In\cite{FFT3}, the link between $EVL$ and {\sl Hitting Time Statistics} (HTS)/ {\sl Returning Time Statistics} (RTS) (for balls) of stochastic processes defined by \eqref{proc} was established and for systems with an absolutely continuous invariant measure, they have been shown to be equivalent to each other, i.e. if such processes have an EVL $G$ then the system has HTS $G$ as well for balls ``centered'' at $\zeta$ and vice versa.  So it is natural to use the observable \eqref{observable}, it gives balls centered at $x_0$. And we have:
\[
 \{M_n \le u_n\} = \{r_{\{X_0 > u_n\}} >n\} = U_n^c
\]

\end{remark}
We may also consider the statistics of multiple returns, which we discuss in  the next subsection.

\subsubsection{Rare events points processes and respective convergence}
Let's introduce some formalism first. Let $\mathcal{W}$ denote the semi-ring of subsets of $\R^{+}$ whose elements are half closed half open intervals $[b,c)$, for $b, c \in \R^{+}$. Let $\mathcal{V}$
be the ring generated by $\mathcal{W}$. For each element $I \in \mathcal{V}$, there exist $k \in \N$ and $k$ intervals $J_1, \cdots, J_k \in \mathcal{W}$ such that $I = \cup_{j=1}^k J_j$. 
For $J=[b,c) \in \mathcal{W}$ and $\alpha \in \R$, we denote $\alpha J : = [\alpha b, \alpha c)$ and
$J+\alpha: = [b +\alpha, c+\alpha)$. Similarly, for $I \in \mathcal{V}$, we define $\alpha I : = \alpha J_1 \cup \cdots \cup \alpha J_k$ and $I+ \alpha :=(J_1+\alpha) \cup \cdots \cup (J_k + \alpha)$. 

\begin{defn}[Rare Event Point Process]
For stationary processes $X_0, X_1, \cdots$ and sequences $(u_n)_{n\in \N}$ satisfying \eqref{rate}, we define the \textsl{Rare Event Point Process}(REPP) by setting: 
\[
N_n(I):= \sum_{j \in a_n I \cap \N_0} \textbf{1}_{\{X_j > u_n\}}~~ \text{for every $I \in \mathcal{V}$}
\]
where $a_n I \in \mathcal{V}$, is the re-scaled time period. 
And $a_n$ is taken to be $1/\mu(X_0 > u_n)$, according to Kac's Theorem, 
the expected waiting time before the occurrence of one exceedance. In fact, this is counting the number of exceedances during the re-scaled time period $a_n I$. 

\end{defn}

\begin{defn}
[Poisson process of intensity $\theta$]
Let $S_1, S_2, \cdots$ be an i.i.d. sequence of random variables with common exponential distribution of mean $1/\theta$. Given this sequence of r.v., for $I \in \mathcal{V}$, set
\[
N(I) = \int \textbf{1}_I d(\sum_{i=1}^{\infty} \delta_{S_1+\cdots+S_i})
\]
where $\delta_l$ denotes the Dirac measure at $l>0$ and, as before, $I \in \mathcal{V}$. We call such $N$ is a \textit{Poisson process of intensity} $\theta$. In special case, when $I = [0, t)$, we also denote $N([0, t))$ by $N(t)$.
\end{defn}

\begin{remark}
If $\theta=1$ then  $N$ is called a standard Poisson process and, for every $t > 0$, the random variable $N(t)$ has a Poisson distribution of mean $t$.
\end{remark}

In \cite[page 4]{FHN},  two conditions $D_3(u_n)$ and $D'(u_n)$ are given  on the dependence structure of a general stationary stochastic process to ensure that the REPP $N_n$ converges in distribution to a standard Poisson process. Also they ensure dependent stationary stochastic processes to have EVL.
These two conditions are the following: 


\textbf{Condition} $(D_3(u_n))$. We say that $D_3(u_n)$ holds for the sequence $X_0, X_1, X_2, \cdots$ if for any integers $l$, $t$ and $n$ 
\[
\bigg\vert\mu(\{X_0 > u_n\} \cap \{M_{t,l}\le u_n\})-\mu(\{X_0 > u_n\})\mu(\{M_l\le u_n\})\bigg\vert \le \gamma(n,t)
\]
where $ M_{t, l}= \max \{X_t, X_{t+1}, \cdots, X_{t+l-1}\}$, and $ \gamma(n,t)$ is nonincreasing in $t$ for each $n$ and $n\gamma(n,t_n)\to 0$ as $n\to \infty$ for some sequence $t_n=o(n)$, $t_n \to \infty$. 


\textbf{Condition} ($D'(u_n)$): The condition $D'(u_n)$ is said to hold for the stationary sequence $\{X_i\}$ and 
the sequence $\{u_n\}$ if 
\[
\limsup_{n \to \infty } n \sum_{j=1}^{[n/k]}\mu(X_0 > u_n, X_j>u_n) \to 0
\]
as $k \to \infty$.


While $D_3(u_n)$ is a condition on the long range dependence structure of the stochastic process $X_0, X_1, \cdots$, condition $D'(u_n)$ is a non-clustering condition, which states that if a large reading 
is observed(say the level $u_n$) at some time $j< n$, then one must wait for a large time $o(n) \to \infty$
 before another reading larger than or equal to $u_n$ is observed.
 
 \textbf{Assumption A}: For $\mu$ a.e. $p \in \Omega$ there exists $\tilde{d} = \tilde{d}(p) >0$ such that if $A_{r,\epsilon}(p)= \{y \in \Omega : r\le d(p, y) \le r+\epsilon\}$ is 
a shell of inner radius $r$ and outer radius $r+\epsilon$ about the point $p$, and if $r$ is sufficient small and $0<\epsilon \ll r < 1$, then $\mu(A_{r,\epsilon}) < \epsilon^{\tilde{d}}$.
 
 \begin{remark}\label{d3exp}
 For systems satisfying Assumption A, condition $D_3(u_n)$ often follows easily  if there are good enough estimates on  decay of correlation for observations in a suitable Banach space.
 \end{remark}

\
J. M. Freitas, N. Haydn, and M. Nicol \cite{FHN} states that the REPP $N_n$ converges in distribution to a standard Poisson process for functions maximized at 
generic points in a variety of billiard systems. They prove this by verifying that the   conditions $D_3(u_n)$ and $D'(u_n)$ hold for such systems. 

\subsection{Main results}
In section \ref{VolPres}, to introduce the main ideas of our analysis of  the Lorenz system in a simpler setting, we establish results of independent interest, namely shrinking target properties 
 for  general skew product maps which do  preserve the two-dimensional Lebesgue measure. We have the following theorem:
\begin{thm}\label{volumepreserving}
Suppose $(\Omega, \mathcal{B}, m_2)$ is a probability space, where $\Omega = I \times I$, with $I = [-\frac{1}{2}, \frac{1}{2}]$, and $m_2$ is the two-dimensional Lebesgue measure. Let $m$ denote the one-dimensional Lebesgue measure.
  Suppose 
$F :  \Omega \to \Omega$ is a map  in the form $F(x,y) = (T(x), G(x,y))$, where $T: I \to I$. Here $F$ 
preserves the two-dimensional Lebesgue measure and $T$ preserves the one-dimensional Lebesgue measure. In addition, 
$T$ satisfies exponential decay of correlation with observables in $BV$ norm versus $L^1$ norm, $i.e$.
\[
|\int \phi \psi \circ T^n d m - \int \phi d m \int \psi d m | \le C \theta^n ||\phi||_{BV} ||\psi||_1
\]
and $F$ satisfies exponential decay of correlation with observables in Lipschitz norm versus Lipschitz norm:
\[
 |\int \phi \psi \circ F^n d m_2 - \int \phi d m_2 \int \psi d m_2 | \le C \alpha^n ||\phi||_{Lip} ||\psi||_{Lip}
\]

Consider nested balls $(B_i(p))$, centered at some $p \in \Omega$, with 
$m_2(B_i(p)) \ge \frac{C}{i^{\gamma_1}}$ for some $\gamma_1 >0$, and $\limsup (\log i) (m_2(B_i))^{\frac{1}{2}} \le C$.
Then we have the Strong Borel Cantelli property
\[
 \frac{S_n(x,y)}{E_n} \to 1 ~a.s.
\]
where $S_n(x,y) = \sum_{j=0}^{n-1} 1_{B_j}\circ F^j(x,y)$, $E_n = \sum_{j=0}^{n-1}m_2(B_j)$. 
\end{thm}


Then in section \ref{Lorenz}, we establish the shrinking target property for the two-dimensional Lorenz map $F$ and Extreme Value Laws for the system with the Lorenz map. When we consider shrinking target balls, they have different shapes according to different metrics. Technically, balls of different shapes are equivalent to each other, but we are able to deal with rectangle balls and circle balls. While rectangle balls consist of local unstable manifold of same length, circle balls don't.  For rectangle balls, we prove the following theorem: 
\begin{thm}\label{BCLsquare}
Consider a sequence of nested square balls ${A_i}$, centered at a  point $p$, of side length $2r(i)$ 
such that $\mu(A_i) \ge \frac{C_2}{i^{\gamma_1}}$, with $ \gamma_1 \ge 0$. Assume that $p$ has a local product structure for 
sufficiently small neighborhoods. Also we assume $(\log i)m (\pi (A_i)_{\gamma})$ is bounded, where $\gamma$ is any local unstable manifold of $A_i$ and $\pi$ is the projection map onto the first dimension. Then if $\sum_i \mu (A_i)$ diverges we have
the Strong Borel Cantelli Property for the squares $A_i$ of sidelength $2r(i)$.
\end{thm}

And for circle balls, we show:
\begin{thm}\label{BCLcircle}
 Consider a sequence of nested circle balls ${A_i}$,centered at a  point $p$, of radius $r(i) = e^{-u_i}$. And assume $\mu(A_i) \ge \frac{C_2}{i^{\gamma_1}}$, with $ \gamma_1 \ge 0$. Also assume that $p$ has a local product structure for 
sufficiently small neighborhoods. 
Then if $\sum_i \mu (A_i)$ diverges we have
the Strong Borel Property for the balls $A_i$.
\end{thm}

We use techniques from subsection \ref{ExValueL}(About EVL) to prove Theorem \ref{BCLcircle}, therefore we will prove it after we establish Extreme Value Laws for the system with the Lorenz map $F$. We also show: 

\begin{thm}\label{EVL}
Consider the dynamical systems $(\Omega, \mathcal{B}, \mu, F)$, where $F$ is the Lorenz map and $F$ preserves the measure $\mu$, whose decomposition on the unstable leaves is absolutely continuous with respect to the one-dimensional Lebesgue measure.  If $\varphi$ is defined as \eqref{observable}, and let $X_n = \varphi \circ F^n$. We assume $x_0$ is not periodic under F. Then $X_n$ satisfies a Type I extreme value law, i.e.
\[
 \lim_{n\to \infty}\mu(M_n \le u_n) = e^{-e^ {-v}}
\]
\end{thm}

In section \ref{flow}, we extend our results to the Lorenz flow.

\section{Volume Preserving Skew Products}\label{VolPres}
As we mentioned in the introduction, the Lorenz map $F$ has a skew product form $F(x,y)=(T(x), G(x,y))$. And $F$  does not preserve the two-dimensional Lebesgue measure  $m_2$ but 
the one-dimensional map  $T$ 
preserves a measure absolutely continuous with respect to  the one-dimensional 
Lebesgue measure $m$, with a Lipschitz density.  In this section, we talk about general skew product maps $\tilde{F}(x,y)=(T(x), G(x,y))$, which do preserve the two-dimensional Lebesgue measure and T does preserve the one-dimensional Lebesgue measure. We stated our Theorem \ref{volumepreserving} in the introduction, now let's prove it.

\begin{proof}[{\bf Proof of Theorem \ref{volumepreserving}}]\mbox{}\\* 
Let $f_k = \textbf{1}_{B_k} \circ F^k(x,y)$, $ E(f_k) = m_2(B_k)$ and let  $a\ge \frac{-7\gamma_1}{\log \alpha}$; we will use $a$ later.  To prove Strong Borel Cantelli Property, 
according to \cite{HNVZ},  it suffices to show the (SP) property, i.e. for all $m<n$
\[
\sum_{i=m}^n \sum_{j=i+1}^n (E(f_i f_j)-E(f_i)E(f_j)) \le C \sum_{i=m}^n E(f_i)
\]

We calculate
\begin{eqnarray*}
E(f_i f_j)&=& \int \textbf{1}_{B_i} \circ F^i(x,y) \cdot \textbf{1}_{B_j} \circ F^j(x,y) d m_2\\
         &=& \int \textbf{1}_{B_i} \cdot \textbf{1}_{B_j} \circ F^{j-i}(x,y) d m_2\\
      &=& m_2 (B_i \cap F^{-(j-i)} B_j)\\
       &\le& C_1 (m_2(B_i))^{\frac{1}{2}} \cdot m(\pi_X B_i \cap T^{-(j-i)} \pi_X B_j)\footnotemark\\
      &= & C_1 (m_2(B_i))^{\frac{1}{2}} \cdot \int  \textbf{1}_{\pi_X B_i} \cdot \textbf{1}_{\pi_X B_j} \circ T^{j-i}(x) d m\\
     &\le& C_1 (m_2(B_i))^{\frac{1}{2}} \cdot (\int\textbf{1}_{\pi_X B_i} \int \textbf{1}_{\pi_X B_j} + C \theta^{j-i} ||\textbf{1}_{\pi_X B_i}||_{BV} ||\textbf{1}_{\pi_X B_j}||_1 )\\
       &\le&  C_1 (m_2(B_i))^{\frac{1}{2}}\cdot \left[  (m_2(B_i))^{\frac{1}{2}} \cdot  (m_2(B_j))^{\frac{1}{2}} + C \theta ^{j-i}  (m_2(B_j))^{\frac{1}{2}}\right]\\
       &\le& C_1 (m_2(B_i))^{\frac{3}{2}} + C \theta ^{j-i} m_2(B_i)
\end{eqnarray*}
\footnotetext{Here $m_2 (B_i \cap F^{-(j-i)} B_j) \le C_1 (m(B_i))^{\frac{1}{2}} \cdot m(\pi_X B_i \cap T^{-(j-i)} \pi_X B_j) $ is because that if $\tilde{x}=(x,y) \in B_i$ and 
$F^{j-i}(\tilde{x}) \in B_j$, then their projection $x \in \pi_X B_i$ and $T^{j-i}(x) \in \pi_X B_j$ so $m_2(B_i \cap F^{-(j-i)} B_j) = \int m((B_i \cap F^{-(j-i)} B_j)_y)dm(y)
=\int m(\{ \tilde{x}= (x,y) | \tilde{x} \in B_i, F^{j-i}(\tilde{x}) \in B_j\})dm(y) \le \int m(x \in \pi_X B_i, T^{j-i}(x) \in \pi_X B_j) m(y)$, where  $\pi_X B_i$ is the projection of the maximal horizontal section of 
the ball $B_i$.}

since $\int\textbf{1}_{\pi_X B_i}dm = m(\pi_X B_i) = (m_2(B_i))^{\frac{1}{2}}$.

Recall $a\ge \frac{-7\gamma_1}{\log \alpha}$, so that
\begin{eqnarray*}
 & &\sum_{j=i+1}^n (E(f_i f_j)-E(f_i)E(f_j))\\
  &\le& ( \sum_{j=i+1}^{i+a\log{i}} + \sum_{j>i + a\log{i}}) [E(f_i f_j)-E(f_i)E(f_j)] \\
&\le& C_1 (\log i) (m_2(B_i))^{\frac{3}{2}}  + C  m_2(B_i) +  \sum_{j>i + a\log{i}}  C \alpha^{j-i} ||\tilde{f}_i||_{Lip} ||\tilde{f}_j||_{Lip} + O(\frac{1}{i^{7/2}}) \\
\end{eqnarray*}
where $\tilde{f}_i$ is a Lipschitz approximation to $f_i$, which is constructed as following: $\tilde{f_i} = \textbf{1}_{B_i}(x)$ if $x \in B_i$, $\tilde{f_i} = 0$ if $d(B_i, x) > 1/i^{3\gamma_1}$, 
$0\le \tilde{f_i}\le 1$ and $||\tilde{f_i}||_{Lip} \le i^{3\gamma_1}$.

Then
\begin{eqnarray*}
 \sum_{j>i + a\log{i}}   \alpha^{j-i} ||\tilde{f}_i||_{Lip} ||\tilde{f}_j||_{Lip} &\le& \sum_{j>i + a\log{i}}  \alpha^{j-i} i^{3\gamma_1} j^{3\gamma_1}\\
                                                                                                                       &=& \sum_{\beta =1}^{\infty} \alpha^{a\log{i}+\beta} i^{3\gamma_1} (i + a\log{i} + \beta)^{3\gamma_1}\\
                                                                                                                       &\le& \alpha^{a\log{i}} i^{3\gamma} C i^{3\gamma_1}\\
                                                                                                                       &\le& \frac{C}{i^{\gamma_1}}\le C m_2(B_i)
\end{eqnarray*}

Since $(\log i) (m_2(B_i))^{\frac{1}{2}} \le C$,  the (SP) property is satisfied.
\end{proof}

\begin{remark}
 (SP) property (Sprindzuk Property) is derived from Gal-Koksma theorem, which is given in the Appendix. Once we have the (SP) property, then the Strong Borel Cantelli Property is established because in the Gal-Koksma theorem, if we take 
 $f_k(\omega) = \textbf{1}_{B_k} \circ F^k(x,y)$, $h_k=g_k = E(f_k) = m_2(B_k)$, dividing both sides of the equation by $\sum_{k=1}^n g_k$, we will have  $\frac{S_n(x,y)}{E_n} \rightarrow 1$ $a.s.$.
\end{remark}



 \section{Lorenz System}\label{Lorenz}
 
In this section,  we present our results on the statistical properties of the Lorenz system, i.e.  Borel Cantelli Lemma and Extreme Value Laws.  Recall that the Lorenz map $F$ does not preserve the two-dimensional Lebesgue measure $m_2$, but preserves an invariant measure $\mu$ which has 
 absolutely continuous conditional measures on local unstable manifolds.
 
 \subsection{Borel Cantelli Lemma}
 
 We let $A_r ( p ) $ denote the square of sidelength $2r$ centered at a point $p$ in the two dimensional space $I\times I$. 
   As a consequence of~\cite[Proposition 2.4]{AMV}, for $\mu$ a.e. $p$, there exists an $r ( p )  >0 $ such  that
 for all $r<r ( p )$,  $ A_r ( p )$ has a local product structure  and in particular $\mu$ a.e. $q\in A_r ( p ) $ has a local unstable manifold $\gamma (q) := W^u_{loc} (q)$ 
 which extends fully across $A_r ( p )$.  The local stable manifolds are arbitrarily long for $\mu$ a.e. $q$. The set of local unstable manifolds $\Gamma=\{\gamma (q)\}$
 partition $A_r (p)$ up to a set of zero $\mu$ measure i.e. $\mu (A_r (p) )=\mu (\cup_{q\in A_r (p)} \gamma (q) \cap A_r (p))$. We will drop the dependence on $q$ and write
 $\Gamma=\{\gamma \}$ for simplicity.
 
 Before we prove Theorem \ref{BCLsquare}, we introduce the notation. 
 
 Given such a point $p$ we let $A$ be a square based at $p$, with sidelength  smaller than $2r(p)$.
 
 Let  $A_{\gamma}= A \cap \gamma$ for $\gamma \in \Gamma$,
 \[
  \mu(A) = \int_I m_{\gamma} (A_{\gamma}) d\nu(\gamma )
 \]
 where $m_{\gamma}$ is the induced  measure of $\mu$ on $\gamma$ and $\nu$ is conditional measure in the decomposition of $\mu$ with respect to  the partition $\{ \gamma \}$.
 
 Let $\pi$ be the projection map onto the first dimension and note that $m_{\gamma} ( A_\gamma) \sim m(\pi A) $.  

 For the Lorenz map $F$ (see \cite[ Theorem 4.7]{GP}), we have exponential decay in Lipschitz versus Lipschitz 
\[
 |\int \phi \psi \circ F^n d \mu - \int \phi d \mu \int \psi d \mu | \le C \alpha^n ||\phi||_{Lip} ||\psi||_{Lip}
\]
and for the base map $T$ (see \cite[Propostion 2.2]{GP} ), we have exponential decay in $L^1$ versus $BV$
\[
|\int \phi \psi \circ T^n d m - \int \phi d m \int \psi d m | \le C \theta^n ||\phi||_{BV} ||\psi||_1
\]
By taking $\phi = \psi = \textbf{1}_{\pi A}= \textbf{1}_{\pi A_{\gamma}}$, we have 
\[
 m(\pi A_{\gamma} \cap T^{-n} (\pi A_{\gamma}) ) - (m (\pi A_{\gamma}))^2 \le C \theta^n ||\textbf{1}_{\pi A_{\gamma}}||_{BV} ||\textbf{1}_{\pi A_{\gamma}}||_1
\]
That is
\[
  m(\pi A_{\gamma} \cap T^{-n} (\pi A_{\gamma}) ) \le (m (\pi A_{\gamma}))^2 + C' \theta^n m(\pi A_{\gamma})
\]
since $||\textbf{1}_{\pi A_{\gamma}}||_{BV}$ is bounded.
 
\begin{proof}[{\bf Proof of Theorem \ref{BCLsquare}}]\mbox{}\\*
 We will establish (SP) property. Without loss of generality, we assume $i < j$. We notice $m(\pi (A_i)_\gamma)$ is equal for all $\gamma \in \Gamma$ since they are square balls. Thus,
\begin{eqnarray*}
 & &\mu(A_i \cap F^{-(j-i)}A_j) \le \mu(A_i \cap F^{-(j-i)} A_i)\\ &\sim& \int_I m(\tilde{x} \in (A_i)_{\gamma}: F^{j-i}(\tilde{x}) \in A_i ) d\nu(\gamma) \\
  &\le& \int_I m(\pi (A_i)_{\gamma} \cap T^{-(j-i)} (\pi (A_i)_{\gamma})) d\nu(\gamma)\\
  &\le& \int_I (m (\pi (A_i)_{\gamma}))^2 + C' \theta^{(j-i)} m(\pi (A_i)_{\gamma}) d\nu(\gamma)\\
  &=& \int_I (m (\pi (A_i)_{\gamma}))^2 d\nu(\gamma) + C' \theta^{(j-i)} \int_Im(\pi (A_i)_\gamma) d\nu(\gamma)\\
  &=& \int_I (m (\pi (A_i)_{\gamma}))^2 d\nu(\gamma) + C' \theta^{(j-i)} \mu(A_i) \\
  &\le& C m (\pi (A_i)_{\gamma})\int_I m(\pi (A_i)_{\gamma}) d\nu(\gamma)+ C' \theta^{(j-i)} \mu(A_i)\\
  &=&C m (\pi (A_i)_{\gamma})\mu(A_i)+C' \theta^{(j-i)} \mu(A_i)\label{sec}
\end{eqnarray*}



Since $\textbf{1}_{A_n}$ is not Lipschitz, we let $\phi_n$ be a  Lipschitz approximation of $\textbf{1}_{A_n}$ such that
\begin{enumerate}
 \item $||\textbf{1}_{A_n}- \phi_n||_1 < (\mu(A_n))^3$
 \item $||\phi_n||_{Lip} < (\mu(A_n))^{-3}$
\end{enumerate}
Then
\begin{eqnarray*}
 & &\left| \int \textbf{1}_{A_i} \cdot \textbf{1}_{A_j} \circ F^{j-i} d \mu-\int \textbf{1}_{A_i} d\mu \int \textbf{1}_{A_j}d\mu\right|\\
 &=&\left| \int ([\textbf{1}_{A_i}-\phi_i]+\phi_i )\cdot ([\textbf{1}_{A_j}-\phi_j]+\phi_j) \circ F^{j-i} d \mu-\int ([\textbf{1}_{A_i}-\phi_i]+\phi_i) d\mu \int ([\textbf{1}_{A_j}-\phi_j]+\phi_j)d\mu\right|\\
 &\le&\left| \int \phi_i \phi_j \circ F^{j-i} d\mu - \int \phi_i d\mu \int \phi_jd\mu \right|+\left| \int (\textbf{1}_{A_i}-\phi_i)(\textbf{1}_{A_j}-\phi_j)\circ F^{j-i}d\mu \right|\\
 & & + \left| \int \phi_i (\textbf{1}_{A_j}-\phi_j)\circ F^{j-i} d\mu\right| +\left| \int(\textbf{1}_{A_i}-\phi_i)\phi_j \circ F^{j-i} d\mu\right| \\
 & & + \left| \int(\textbf{1}_{A_i}-\phi_i)d\mu \int(\textbf{1}_{A_j}-\phi_j)d\mu\right| + \left| \int \phi_i d\mu \int (\textbf{1}_{A_j}-\phi_j)d\mu\right|\\
 & & + \left| \int(\textbf{1}_{A_i}-\phi_i)d\mu \int \phi_j  d\mu\right|\\
 &\le& C \alpha^{j-i}||\phi_i||_{Lip}||\phi_j||_{Lip}+ \tilde{C}(\mu(A_i))^3\\
\end{eqnarray*}


If we choose $a \ge \frac{-7\gamma_1}{\log{\alpha}}$, then
\begin{eqnarray*}
 \sum_{j>i + a\log{i}}   \alpha^{j-i} ||\phi_i||_{Lip} ||\phi_j||_{Lip} &\le& \sum_{j>i + a\log{i}}  \alpha^{j-i} i^{3\gamma_1} j^{3\gamma_1}\\
                                                                                                                       &=& \sum_{\beta =1}^{\infty} \alpha^{a\log{i}+\beta} i^{3\gamma_1} (i + a\log{i} + \beta)^{3\gamma_1}\\
                                                                                                                       &\le& \alpha^{a\log{i}} i^{3\gamma_1} C i^{3 \gamma_1}\\
                                                                                                                       &\le& \frac{C}{i^{\gamma_1}}\\
                                                                                                                       &\le& C \mu(A_i)
\end{eqnarray*}

Thus, 
 let  $f_k = \textbf{1}_{A_k} \circ F^k(x,y)$, $E(f_k) = \mu(A_k)$, so that we have
\begin{eqnarray*}
 & & \sum_{j=i+1}^n (E(f_i f_j)-E(f_i)E(f_j)) \\
 &=& \sum_{j=i+1}^n \mu(A_i \cap F^{j-i} A_j) - \mu(A_i)\mu(A_j)\\
 &=&  \sum_{j=i+1}^{i+a \log i} \left[ \mu(A_i \cap F^{j-i} A_j) - \mu(A_i)\mu(A_j)\right]+  \sum_{j>i+a\log i} \left[ \mu(A_i \cap F^{j-i} A_j) - \mu(A_i)\mu(A_j)\right]\\\\
 &\le& \sum_{j=i+1}^{i+a\log i}  \left[m (\pi (A_i)_{\gamma})\mu(A_i)+C' \theta^{j-i} \mu(A_i)\right]+ C\mu(A_i)\\
 &\le&  C (\log i)m (\pi (A_i)_{\gamma})\mu(A_i)+ C_1 \mu(A_i)+ C_2\mu(A_i)\le \tilde{C} \mu(A_i)
\end{eqnarray*}
 We have established the  (SP) property and thus the strong Borel Cantelli lemma for $\{A_i\}$.
 
\end{proof}

For more general case, i.e. circle balls, we have Theorem \ref{BCLcircle}. As we mentioned in the introduction,  the proof of Theorem \ref{BCLcircle} needs the techniques from the subsection on EVL so we will do the proof then. 


\subsection{Extreme Value Laws}\label{ExValueL}

In this section, we establish EVL  for Lorenz maps by essentially showing that the two conditions $D_3(u_n)$ and $D'(u_n)$, which were introduced in section \ref{evl}, are satisfied. Recall:

\textbf{Condition} $(D_3(u_n))$. We say that $D_3(u_n)$ holds for the sequence $X_0, X_1, X_2, \cdots$ if for any integers $l$, $t$ and $n$ 
\[
\bigg\vert\mu(\{X_0 > u_n\} \cap \{M_{t,l}\le u_n\})-\mu(\{X_0 > u_n\})\mu(\{M_l\le u_n\})\bigg\vert \le \gamma(n,t)
\]
where $ M_{t, l}= \max \{X_t, X_{t+1}, \cdots, X_{t+l-1}\}$, and $ \gamma(n,t)$ is nonincreasing in $t$ for each $n$ and $n\gamma(n,t_n)\to 0$ as $n\to \infty$ for some sequence $t_n=o(n)$. 

\vspace{.5cm}

\textbf{Condition} ($D'(u_n)$): The condition $D'(u_n)$ is said to hold for the stationary sequence $\{X_i\}$ and 
the sequence $\{u_n\}$ if 
\[
\limsup_{n \to \infty } n \sum_{j=1}^{[n/k]}\mu(X_0 > u_n, X_j>u_n) \to 0
\]
as $k \to \infty$.

\vspace{.5cm}

Before we prove Theorem \ref{EVL}, let's prove the following two lemmas:
\begin{lemma}\label{AssumptionA}
 Suppose we have a local product structure about a point $x_0$ and the local dimension exists, denoted by $d$. Then  Assumption A is satisfied.
\end{lemma}

\begin{proof}
 As before, the conditional measure $\mu_{\gamma}$  is equivalent to Lebesgue in the local unstable direction, and $r$ is small, i.e., $r<1$.
 Let $\epsilon = r^w$, with $w>1$.
 We need to prove that the measure of the annular region $S= A_{r+\epsilon} (x_0) / A_r (x_0)$ is small.

We decompose $\mu$ in a neighborhood of $x_0$ as follows 
$$\mu (A)=\int_{\gamma \in \Gamma} m(\gamma \cap A) d\nu (\gamma)$$
where $\gamma$ is the foliation into local unstable manifolds. Since we have a local product structure at $x_0$, these extend all the way across
a sufficiently small rectangular neighborhood of $x_0$.

Now consider the equation of the circles $x^2+y^2=r^2$ and $x^2+y^2=(r+\epsilon)^2=r^2 +2r^{w+1}+r^{2w}$.
The larger  circle contains some local unstable manifolds which are not in the smaller circle but the greatest
length of  these is found by setting $y^2=r^2$ in the second equation and solving for $\delta x \le  r^{\frac{w+1}{2}}$. Their
length is less than $ r^{\frac{w+1}{2}}$, so that 
\[
 \mu (S) \le \int_\Gamma (S\cap \gamma)d\nu (\gamma) < r^{\frac{w+1}{2}} < \epsilon^{\frac{w+1}{2w}}<\epsilon^{1/2}
\]

\end{proof}

\begin{lemma}\label{annulus}
 \begin{enumerate}[(a)]
  \item For $\mu$ a.e. $x_0$, for every $\epsilon > 0$, there exists an $N \in \N$ such that for all $n\ge N$
  \[
   \frac{1}{d+\epsilon}(v+\log n) \le u_n(v) \le \frac{1}{d-\epsilon}(v+\log n)
  \]
  where $d$ is the local dimension.
 \item Denote by $S(n,x_0) = A_{e^{-u_n}}(x_0)/A_{e^{-u_n}-e^{-u_n^2}}(x_0)$, the annulus region between balls centered at $x_0$ of radius $e^{-u_n}$ and $e^{-u_n}-e^{-u_n^2}$. There exists $\delta = \delta(x_0) \in (0,1)$ such that for n large enough 
 \[
  \mu(S(n,x_0)) \le C_3 n ^{-2\delta \upsilon-\delta \log n}
 \]
\end{enumerate}
\end{lemma}

\begin{proof}
 \begin{enumerate}[(a)]
  \item By the definition of the local dimension, for any $\epsilon>0$, there is an $N$ such that for all $n\ge N$, $(e^{-u_n})^{(d+\epsilon)} \le \mu(U_n) \le (e^{-u_n})^{(d-\epsilon)}$, and $\mu(U_n) = e^{-v}/n$. We get immediately
  \[
    \frac{1}{d+\epsilon}(v+\log n) \le u_n(v) \le \frac{1}{d-\epsilon}(v+\log n)
  \]
  \item According to Subsection \ref{localdim} and \ref{localstr}, Lorenz system has local dimension and local structure, and by Lemma \ref{AssumptionA}, Assumption A is satisfied for Lorenz system. So there exists a $\delta \in (0,1)$ such that
  \begin{eqnarray*}
    \mu(S(n,x_0)) &\le& C (e^{-(u_n^2)})^\delta \\
                             &= &C e^{-(u_n^2)\delta} \\
                             &\le&C \exp\left(-\frac{\delta}{(d+\epsilon)^2}(v+\log n)^2\right)\\ 
                             &\le& C_3n ^{-2\delta' \upsilon-\delta' \log n}
  \end{eqnarray*}
\end{enumerate}

\end{proof}

\begin{proof}[{\bf Proof of Theorem \ref{EVL}}]\mbox{}\\*
To show $EVL$ for the Lorenz system, it suffices to show $D_3(u_n)$ and $D'(u_n)$.
Then, as mentioned in Remark \ref{d3exp}, $D_3(u_n)$ is easily proven if the system satisfies Assumption A and good enough estimates for decay of correlations. We already know Assumption A is satisfied for Lorenz systems, and we have exponential decay of correlation for Lorenz map
\[
 |\int \phi \psi \circ F^n d \mu - \int \phi d \mu \int \psi d \mu | \le C \alpha^n ||\phi||_{Lip} ||\psi||_{Lip}
\] 
So we prove $D_3(u_n)$ in the following paragraph.

By part (b) of Lemma \ref{annulus}, we have $\mu(S(n,x_0)) \le C_3 n ^{-2\delta \upsilon-\delta \log n}$, where $\delta = \delta(x_0) \in (0,1)$ for $n$ large enough and $\upsilon$ could be any number. 
Take $\phi_n$ to be the Lipschitz approximation of $\textbf{1}_{U_n}=\textbf{1}_{\{X_0 >u_n\}}$ such that $\phi_n(x) = 1$ if $x$ inside $A_{e^{-u_n}-e^{-u_n^2}}(x_0)$, $\phi_n=0$ if $x$ is outside $U_n$, and decays to 0 at a linear rate on $S(n,x_0)$.
So we have the estimate $||\phi_n - \textbf{1}_{\{X_0 > u_n\}}||_1<\mu(S(n,p))$ and $||\phi_n||_{Lip} \le e^{-u_n^2}$. Also let
$\psi_n = \textbf{1}_{\{M_l \le u_n\}}$. By (\cite{GHN}, Lemma 3.1), we then have 
\begin{eqnarray*}
 |\int \phi_n \psi_n \circ F^t d \mu - \int \phi_n d \mu \int \psi_n d \mu | \le 
 O(1)(||\phi_n||_{\infty} \tau_1^{\lfloor t/2 \rfloor}+||\phi_n||_{Lip}\alpha^{\lfloor t/2 \rfloor})
\end{eqnarray*}
and
\begin{eqnarray*}
& &|\mu(\{X_0 >u_n\} \cap \{M_{t,l} \le u_n\}) - \mu(\{X_0 >u_n\}\mu(\{M_l \le u_n\}))| \\ 
&\le& |\int (\textbf{1}_{\{X_0 >u_n\}}-\phi_n)\psi_n \circ F^t d\mu|+|\int \phi_n \psi_n \circ F^t d \mu - \int \phi_n d \mu \int \psi_n d \mu |\\& &+|\int(\textbf{1}_{\{X_0 >u_n\}}-\phi_n)d\mu \int \psi_n d\mu|\\
&\le& O(1)( n ^{-2\delta \upsilon-\delta \log n} + ||\phi_n||_{\infty} \tau_1^{\lfloor t/2 \rfloor}+||\phi_n||_{Lip}\alpha^{\lfloor t/2 \rfloor})
\end{eqnarray*}
where $\tau_1$ is from \cite[Proposition 1.1]{GHN}. Let $\gamma(n,t)=  n ^{-2\delta \upsilon-\delta \log n}+ ||\phi_n||_{\infty} \tau_1^{\lfloor t/2 \rfloor}+||\phi_n||_{Lip}
\alpha^{\lfloor t/2 \rfloor}$, and $||\phi_n||_{Lip} \le e^{-u_n^2} \le O(1) n ^{-2\tilde{\delta} \upsilon-\tilde{\delta} \log n}$(by similar argument we did to get $\mu(S(n,x_0))$). Take $t = t_n = (\log n)^5$, 
so that $n\gamma(n,t_n) \to 0$ as $n \to \infty$. Thus $D_3(u_n)$ is established.


We now establish $D'(u_n)$. Note $F^{-j} (U_n) = \{ X_j > u_n\}$, then
\[
\mu(X_0 >u_n, X_j > u_n)=\mu(U_n \cap F^{-j} U_n) \le \int_{\gamma \in \Gamma}m(\pi (U_n)_\gamma \cap T^{-j} (\pi (U_n)_\gamma))d\nu(\gamma)
\]
Where $\pi$ is the projection map onto the first dimension. For each leaf $\gamma \in \Gamma$, define $B_{r_n,\gamma} = \pi (U_n)_\gamma= B(q_\gamma e^{-u_n}, \pi(x_0))$, where $r_n = e^{-u_n}$, $0 \le q_\gamma \le 1$.

Consider points where the local unstable manifold is less  than $r^3$, so that the integral splits as follows 
\begin{eqnarray*}
 \int_{\gamma \in \Gamma}m(\pi (U_n)_\gamma \cap T^{-j} (\pi (U_n)_\gamma))d\nu(\gamma)&=& \int_{\gamma \in \Gamma_1}m(\pi (U_n)_\gamma \cap T^{-j} (\pi (U_n)_\gamma))d\nu(\gamma)\\
& &+\int_{\gamma \in \Gamma_2}m(\pi (U_n)_\gamma \cap T^{-j} (\pi (U_n)_\gamma))d\nu(\gamma) 
\end{eqnarray*}
where $\Gamma_1$ is the set of local unstable manifolds which has length less than $r^3$, and $\Gamma_2 = \Omega / \Gamma_1$. The reason for doing so is because if the local unstable manifold has a short length, the point in the projection probably has no short return. 

For $\gamma \in \Gamma_2$, define
\[
E_{k, \gamma} =\{x\in B_{r_k, \gamma}:\,d(T^jx,x)<\frac{1}{k^{1/3}},\,\textrm{for some}\:1\leq j\leq (\log k)^5\}.
\]
By \cite[Proposition 4.2]{GHN},  there exists $0<a<1,0< \tilde\theta<1$ such that
$$m(E_{k, \gamma})<\tilde\theta^{(\log k^{1/3})^{a}}$$
We only need $a <1/2$, so we take $a=1/3$.

Let $0<\beta \le \frac{1}{2}$ and let $0<\rho<1$ such that $\rho\beta < \beta/3$. 

Define the set 
\begin{equation}
 F_{k,\gamma}:=\{m(B_{q_\gamma \exp(-k^\beta)}(x)\cap E_{\exp(k^{3\beta})})\ge m(B_{q_\gamma \exp(-k^\beta)}(x))\exp(-k^{\beta\rho}) \}.
\end{equation}
If $x\in F_{k,\gamma}$ then 
\begin{equation}\label{eq:F_k.def.1}
 \frac{m(B_{q_\gamma \exp(-k^\beta)}(x)\cap E_{\exp(k^{3\beta})})}{m(B_{q_\gamma \exp(-k^\beta)}(x))}\ge \exp(-k^{\beta\rho});
\end{equation}

If we define 
\[
 M_l(x):=\sup_{r>0}\frac{1}{m(B_r (x))}\int_{B_r (x)} \textbf{1}_{E_l}(y)dm(y)
\]
we see immediately from the definition of $M_l(x)$ and \eqref{eq:F_k.def.1} that for every $x\in F_{k,\gamma}, M_{e^{k^{3\beta}}}(x)\ge e^{-k^{\beta\rho}}$. Hence
\begin{equation}\label{eq:F_k.def.2}
 F_{k,\gamma}\subset \{M_{e^{k^{3\beta}}}(x)\ge e^{-k^{\beta\rho}}\}.
\end{equation}
A theorem of Hardy and Littlewood \cite[Theorem 2.19]{mattila} implies that
\[
 m(|M_l|>c)\le \frac{\|\textbf{1}_{E_l}\|_1}{c};
\] 
$||\cdot||_1$ is with respect to the one-dimensional Lebesgue measure.
As $m(E_l)\le \C\tilde\theta^{(\log l^{1/3})^{1/3}}$ (recall $a=1/3$), 
\[
 m(F_{k,\gamma})\le \C m(E_{\exp(k^{3\beta})})e^{k^{\beta\rho}}\le  \C(e^{\alpha k^{\beta/3}+k^{\beta\rho}})
\] 
where $\alpha:= \log \tilde{\theta}$ and $k$ is large enough. Since $\beta/3 > \beta\rho$,  $\sum_{k>0}m(F_{k,\gamma})<\infty$.
 By the  Borel Cantelli lemma, $m(\limsup F_{k,\gamma}) = 0$, and hence for $m$ almost every $x$ there exists an $N_x$ such that for all $k\ge N_x$, $x\notin F_{k,\gamma}$ for each $\gamma$.

Let $x_0$ be such a generic point, and let $N_{x_0}$ be the corresponding index beyond which $x_0$ does not belong to any $F_{k,\gamma}$.
Since $\lim_{k\to \infty} e^{(k+1)^{\beta} }e^{-k^{\beta}}=1$  the fact that we restricted to a subsequence is of no consequence,
and we obtain the following estimate for all $n$ sufficiently large.  If $1\le j\le (\log n)^5$, then 
\begin{equation}\label{shortreturn}
m\left(B_{r_n, \gamma}\cap T^{-j} B_{r_n,\gamma} \right)\le  m(B_{r_n, \gamma})\exp(-u_n^{\rho}) 
\end{equation}
Summing over $1\le j\le (\log n)^5$ and taking limits as $n\to\infty$ we obtain: 
\begin{eqnarray*}
 & &n \sum_{1}^{(\log n)^5} \int_{\gamma \in \Gamma_2}m(\pi (U_n)_\gamma \cap T^{-j} (\pi (U_n)_\gamma))d\nu(\gamma) \\
 &\le&  n \sum_{1}^{(\log n)^5} e^{-u_n^\rho}\int_{\gamma \in \Gamma_2}  m(\pi (U_n)_\gamma)d\nu(\gamma) \\
 &\le & n \sum_{1}^{(\log n)^5} e^{-u_n^\rho} \mu(U_n) \\
 & = & (\log n)^5 e^{-u_n^\rho} e^{-v} \to 0
\end{eqnarray*}
since $u_n$ has estimates in part(a) of Lemma \ref{annulus}.

And for $\Gamma_1$, 
\[
  n \sum_{1}^{(\log n)^5} \int_{\gamma \in \Gamma_1}m(\pi (U_n)_\gamma \cap T^{-j} (\pi (U_n)_\gamma))d\nu(\gamma) \le n (\log n)^5 e^{-3u_n} \to 0 
\]
Consequently we have 
\[
 n \sum_{1}^{(\log n)^5} \mu ( X_0 >  u_n, X_0 \circ F^j > u_n) \to 0
\]

Finally, similarly to the argument in the case of  Planar Dispersing Billiard Maps in \cite[section 4.1.3]{GHN}, we use  exponential decay of correlations to  show
 \[
 \lim_{n\rightarrow \infty} n \sum_{(\log n)^5}^{p=\sqrt{n}} \mu ( X_0 >  u_n, X_0 \circ F^j > u_n)=0.
\]

\end{proof}

Now we can prove Theorem \ref{BCLcircle}.

\begin{proof}[{\bf Proof of Theorem \ref{BCLcircle}}]\mbox{}\\*
We follow the steps of proof of Theorem \ref{BCLsquare}, and use the techniques from the proof of Theorem \ref{EVL}, we have:

\begin{eqnarray*}
 & &\mu(A_i \cap F^{-(j-i)}A_j) \le \mu(A_i \cap F^{-(j-i)} A_i)\\ &\sim& \int_\Gamma m(\tilde{x} \in (A_i)_{\gamma}: F^{j-i}(\tilde{x}) \in A_i ) d\nu(\gamma) \\
  &\le& \int_\Gamma m(\pi (A_i)_{\gamma} \cap T^{-(j-i)} (\pi (A_i)_{\gamma})) d\nu(\gamma)\\
  &=& \int_{\Gamma_1} m(\pi (A_i)_{\gamma} \cap T^{-(j-i)} (\pi (A_i)_{\gamma})) d\nu(\gamma)\\
  & &+\int_{\Gamma_2} m(\pi (A_i)_{\gamma} \cap T^{-(j-i)} (\pi (A_i)_{\gamma})) d\nu(\gamma)\\
\end{eqnarray*}
where $\Gamma_1$ is the set of local unstable manifolds which has length less than $r^3$, and $\Gamma_2 = \Omega / \Gamma_1$. Thus,
\[
\int_{\Gamma_1} m(\pi (A_i)_{\gamma} \cap T^{-(j-i)} (\pi (A_i)_{\gamma})) d\nu(\gamma) \le r_i^3 \le \mu(A_i)^{\frac{3}{d+\eta}}
\]
and by \eqref{shortreturn}
\[
\int_{\Gamma_2} m(\pi (A_i)_{\gamma} \cap T^{-(j-i)} (\pi (A_i)_{\gamma})) d\nu(\gamma) \le \mu(A_i)\exp(-u_i^{\rho})
\] 
Let  $f_k = \textbf{1}_{A_k} \circ F^k(x,y)$ and $E(f_k) = \mu(A_k)$, so that
\begin{eqnarray*}
 & & \sum_{j=i+1}^n (E(f_i f_j)-E(f_i)E(f_j)) \\
 &=& \sum_{j=i+1}^n \mu(A_i \cap F^{j-i} A_j) - \mu(A_i)\mu(A_j)\\
 &=&  \sum_{j=i+1}^{i+(\log i)^5} \left[ \mu(A_i \cap F^{j-i} A_j) - \mu(A_i)\mu(A_j)\right]+  \sum_{j>i+(\log i)^5} \left[ \mu(A_i \cap F^{j-i} A_j) - \mu(A_i)\mu(A_j)\right]\\\\
 &\le& \sum_{j=i+1}^{i+(\log i)^5}  \left[\mu(A_i)^{\frac{3}{d+\eta}}+\mu(A_i)\exp(-u_i^{\rho}) \right]+ C\mu(A_i)\\
 &\le&   (\log i)^5 \mu(A_i)^{\frac{3}{d+\eta}-1}\mu(A_i)+ C_1 \mu(A_i)+ C\mu(A_i)\le \tilde{C} \mu(A_i)
\end{eqnarray*}
where $0<\rho<1/3$. Therefore we have the (SP) property, and the Strong Borel Cantelli property then follows.
\end{proof}

\section{Lorenz Flow}\label{flow}

Let $M$ be the Riemannian manifold, associated with Lorenz flows,  endowed with a metric $d_M$, and $f_t: M \to M$ the Lorenz $C^1$-flow. $\Omega \subset M$ is a 
transverse cross-section of the flow which is a $C^1$-submanifold with boundary, as we stated in previous sections. We know $F: \Omega^* \to \Omega$  preserves a probability measure $\mu$, where $\Omega= [-1/2, 1/2] \times [-1/2, 1/2]$ and $\Omega^* = ([-1/2, 1/2]\backslash\{0\}) \times [-1/2,1/2]$. 
Let $h: \Omega \to \R_+$ be the 
first return time of the flow to $\Omega$,
and $h \notin L^1(\mu)$. Consider the suspension space
\[
 \Omega^h = \{(p,u) \in \Omega \times \R~ | ~0 \le u \le h(p)~\}/ \sim ,~~~\text{where}~~ (p, h(p)) \sim (F(p), 0) 
\]
We model the flow $f_t: M \to M$ in the standard way by the 
suspension flow $\tilde{f}_t: \Omega^h \to \Omega^h$, $\tilde{f}_t(p,u) = (p, u+t)/ \sim$. Denote the 
metric on $\Omega$ by $d_\Omega$, and we define a metric $d_{\Omega^h}$ on $\Omega^h$  by 
\[
d_{\Omega^h} ((p,u),(q,v)) = \sqrt{d_{\Omega}(p,q)^2+|u-v|^2} 
\]
Then we can introduce a projection map
$\pi_M : \Omega^h \to M$, $(p,t) \mapsto f_t(p)$, which is a local $C^1$-diffeomorphism. $\mu$ is an invariant ergodic probability measure for the first return 
map, i.e. our Lorenz map, $F: \Omega^* \to \Omega$. This induces (in the standard way) an invariant measure $\mu^h$, on the suspension $\Omega^h$, which is given by $d\mu \times dm/\bar{h}$ and $\bar{h} = \int_{\Omega} h d\mu$. 
Then $\mu^h$  
determines a $f_t$-invariant measure $\mu_M$ on $M$ by $\mu_M(A) = \mu^h(\pi_M^{-1}A)$ for measurable sets $A$. 

Consider a measurable observation $\varphi: \Omega^h \to \R$ such that $\varphi(x) = -\log d_{\Omega^h}(x, x_0)$ where $x_0$ is any point in $\Omega^h$, then $\varphi$ has a logarithmic singularity at $x_0$.  Define $\Phi: \Omega \to \R$ by
\[
 \Phi(p) := \max \{ \varphi(f_s(p)) ~|~ 0 \le s < h(p) \}
\]
Denote
\[
\varphi_t (p) := \max \{\varphi(f_s(p))|0\le s < t\}
\]
\[
\Phi_N(p) := \max\{\Phi(F^k(p))|0 \le k < N\}
\]
Then we have our main theorem in the flow case: 
\begin{thm}
Assume that $F$ is the Lorenz map and $f_t$ is the corresponding Lorenz flow. Assume the levels $\{u_n\}$ satisfy
\[
n \mu(\Phi_0 > u_n) \to e^{-v}
\] 
which gives some normalizing constants $a_n >0$ and $b_n$ such that
\[
\lim_{\epsilon\to 0}\limsup_{n \to \infty} a_n |b_{[n+\epsilon n]}-b_n|=0
\]
\[
\lim_{\epsilon\to 0}\limsup_{n \to \infty} \left|1 - \frac{a_{[n+\epsilon n]}}{a_n} \right| = 0
\]
Then $\Phi_N$ satisfies a Type I extreme value law, i.e.  
\[
a_N(\Phi_N-b_N) \rightarrow_d e^{-e^{-v}}
\]
implies that
$\varphi_t$ also satisfies a Type I extreme value law, 
\[
a_{\lfloor T/\bar{h} \rfloor}(\varphi_T-b_{\lfloor T/\bar{h} \rfloor}) \rightarrow_d e^{-e^{-v}}
\]
\end{thm}

\begin{proof}
It is a consequence of \cite[Theorem 2.6]{HNT} and Theorem \ref{EVL}. \cite[Sublemma 4.18]{HNT} takes care of the issue that $h$ is not bounded. 
\end{proof}

\section{Appendices}

\subsection{Gal-Koksma Theorem.}
We  recall the following result of Gal and Koksma as formulated by  W. Schmidt~\cite{W1,W2} and stated
by Sprindzuk~\cite{Sprindzuk}:

  Let $(\Omega,\mathcal{B},\mu)$ be a probability space and let $f_k
  (\omega) $, $(k=1,2,\ldots )$ be a sequence of non-negative $\mu$
  measurable functions and $g_k$, $h_k$ be sequences of real numbers
  such that $0\le g_k \le h_k \le 1$, $(k=1,2, \ldots,)$.  Suppose
  there exists $C>0$ such that
  \begin{equation} 
    \tag{$*$} \int \left(\sum_{m<k\le n}( f_k (\omega) - g_k)
    \right)^2\,d\mu \le C \sum_{m<k \le n} h_k
  \end{equation}
  for arbitrary integers $m <n$. Then for any $\epsilon>0$
  \[
  \sum_{1\le k \le n} f_k (\omega) =\sum_{1\le k\le n} g_k   +
  O (\Theta^{1/2} (n) \log^{3/2+\epsilon} \Theta (n)
  )
  \]
  for $\mu$ a.e.\ $\omega \in \Omega$, where $\Theta (n)=\sum_{1\le k
    \le n} h_k$.

\bibliographystyle{plain}
\bibliography{lorenz3}

\end{document}